\documentclass[12pt]{amsart}

\usepackage{amsmath,amssymb,txfonts}
\usepackage{amssymb}
\usepackage{amsxtra}
\usepackage{amsthm, color}
\usepackage{txfonts}
\usepackage{graphicx}
\usepackage{color}
\usepackage{times}

\usepackage[cp1252]{inputenc}

\usepackage{mathrsfs}
\usepackage{graphicx}

\usepackage[active]{srcltx}

\newcommand{\rmv}[1]{}

\numberwithin{equation}{section}

\newcommand{\HH}{\mathcal{H}}
\newcommand{\HK}{\mathcal{K}}
\newcommand{\HM}{\mathcal{M}}
\newcommand{\HE}{\mathcal{E}}

\newcommand{\HD}{\mathcal{D}}
\newcommand{\HS}{\mathcal{S}}
\newcommand{\HB}{\mathcal{B}}
\newcommand{\HN}{\mathcal{N}}

\newcommand{\D}{\mathbb{D}}

\newcommand{\C}{\mathbb{C}}

\newcommand{\T}{\mathbb{T}}

\newcommand{\ran}{\mathrm{ran \ }}

\theoremstyle{plain}
\newtheorem{theorem}{Theorem}[section]
\newtheorem{lemma}[theorem]{Lemma}

\newtheorem{prop}[theorem]{Proposition}
\newtheorem{corollary}[theorem]{Corollary}

\newtheorem*{thma}{Theorem A}
\newtheorem*{thmb}{Theorem B}
\newtheorem*{thmc}{Theorem C}

\theoremstyle{definition}
\newtheorem{example}[theorem]{Example}

\begin{document}

\date{}

\author{Hongxin Chen}
\address{School of Mathematics, Hunan University, Changsha, 410082, PR China}
\email{hongxinchen@hnu.edu.cn}

\author{Caixing Gu}
\address{Department of Mathematics, California Polytechnic State University, San Luis Obispo, CA 93407, USA}
\email{cgu@calpoly.edu}

\author{Shuaibing Luo}
\address{School of Mathematics, Hunan University, Changsha, 410082, PR China}
\email{sluo@hnu.edu.cn}


\title[Schur functions with mate]{De Branges-Rovnyak spaces generated by row Schur functions with mate}

\begin{abstract}
In this paper, we study the de Branges-Rovnyak spaces $\HH(B)$ generated by row Schur functions $B$ with mate $a$. We prove that the polynomials are dense in $\HH(B)$, and characterize the backward shift invariant subspaces of $\HH(B)$. We then describe the cyclic vectors in $\HH(B)$ when $B$ is of finite rank and $\dim (aH^2)^\perp < \infty$.
\end{abstract}
\maketitle


\section{Introduction}
Suppose $A$ is a contraction from a Hilbert space $\HH$ to a Hilbert space $\HK$. Let $\HM(A)$ be the operator range of $A$ with the Hilbert space structure that makes $A$ a coisometry. Let $D_{A^*} = (I-AA^*)^{1/2}$ be the defect operator of $A^*$, and let $\HH(A)$ be $\HM(D_{A^*})$. We call $\HH(A)$ the de Branges-Rovnyak space.

Let $\HD$ and $\HE$ be separable Hilbert spaces, let $H^2(\HD)$ be the Hardy space of square summable power series with coefficients in $\HD$ on the unit disc $\D$ and let $\HS(\HD, \HE)$ be the Schur class functions, i.e. the collection of functions that are analytic on $\D$ and take values in the contractive operators in $\HB(\HD,\HE)$. Suppose $B \in \HS(\HD, \HE)$ is a Schur function. Let $T_B$ be the multiplication operator with symbol $B$ from $H^2(\HD)$ to $H^2(\HE)$. We denote the de Branges-Rovnyak spaces $\HH(T_B)$ and $\HH(T_B^*)$ by $\HH(B)$ and $\HH(B^*)$, respectively. It is known that $\HH(B)$ has the reproducing kernel
$$K^B_w(z)= \frac{I_{\HE}-B(z)B(w)^*}{1-z\overline{w}}.$$
If $L$ is the backward shift
$$Lf(z)=\frac{f(z)-f(0)}{z},$$
then $L$ acts contractively on $\HH(B)$, and $\HH(B)$ is contractively contained in $H^2(\HE)$. Note that if $B(0)=0$, then $K^B_0(z)=I_{\HE}$, hence the constant functions are contained in $\HH(B)$, $\HE\subseteq \HH(B)$, and in fact $\ker L=\HE$. We refer to \cite{BallBoloBasics, dBR66, FM16, FM162, Sa94} for the background of the de Branges-Rovnyak spaces.

Extensive researches have been done on $\HH(B)$ where $B$ is a scalar Schur
function in the last several decades, though many interesting questions
remain open (\cite{FM162}). The research on $\HH(B)$ where $B$ is an operator-valued
Schur function is far less for many reasons (\cite{BallBoloBasics}). For example, the function
theory of operator-valued Schur functions is much more complicated and
operator-valued functions do not commute. Inspired by some recent works on
$\HH(B)$ for operator-valued $B$ (\cite{AM19, LGR}) and matrix-valued Aleksandrov-Clark
measure (\cite{LMT}), we aim to extend several fundamental results of $\HH(B)$ for
scalar $B$ by Sarason (\cite{Sa86, Sa88}) and recent results on cyclic vectors of $\HH(B)
$ for scalar $B$ by Bergman (\cite{Berg}) to the $\HH(B)$ for $B$ a row vector-valued
Schur function. Since our $B\in S(\HD,\C)$ where $\HD$ could be infinite
dimensional, our approach also extends some results of \cite{AM19} where $\HD$ is
assumed to be finite dimensional.

The main results in this paper are the following three theorems.
\begin{thma}
Let $B \in \HS(\HD, \C)$ be such that $\log (1- BB^*) \in L^1(\T)$. Then the polynomials are dense in $\HH(B)$.
\end{thma}

\begin{thmb}
Let $B \in \HS(\HD, \C)$ be such that $\log (1- BB^*) \in L^1(\T)$. Then any proper $L$-invariant subspace of $\HH(B)$ is of the form
$$\HH(B) \cap K_\theta,$$
where $\theta$ is an inner function and $K_\theta = H^2 \ominus \theta H^2$.
\end{thmb}

\begin{thmc}
Let $B \in \HS(\HD, \C)$ be such that $\log (1- BB^*) \in L^1(\T)$, and $a$ the mate of $B$. Suppose $\dim \HD < \infty, \dim (aH^2)^\perp < \infty$, and $\overline{\lambda_1}, \overline{\lambda_2}, \ldots, \overline{\lambda_s}$ are the distinct eigenvalues of $M_z^*$ on $(aH^2)^\perp$. Then each $\overline{\lambda_j}$ is in $\T$, and $f \in \HH(B)$ is cyclic in $\HH(B)$ if and only if $f$ is outer and $f(\lambda_j) \neq 0$, $j = 1, 2, \ldots, s$.
\end{thmc}
When $\dim \HD = 1$, Theorem A and Theorem B are foundational results in the theory
of $\HH(b)$ spaces which were proved in 1986 by Sarason in \cite{Sa86}. When $\dim
\HD<\infty ,$ Theorem A and Theorem B were obtained in 2019 by Aleman and Malman in \cite{AM19} through some deep results about the $\HH(B)$ spaces established in their paper. When $\dim \HD=1,$ Theorem C
were shown by Bergman in \cite{Berg} recently. We will prove Theorem A and B in
Section 2, and prove Theorem C in Section 3. In Section 3, we also study when $B
$ satisfies the Carath\'{e}odory condition by using some idea of
Aleksandrov-Clark measure associated with $B$ from \cite{LMT}. Our results
complement some results in \cite{LMT} where $B$ is matrix-valued. Along the way, we
also extend some results of Sarason \cite{Sa88} on angular derivatives of functions
in $\HH(B)$ and reproducing kernel of $\HH(B)$ on the circle.

\section{Schur functions having mate}
\subsection{A characterization of de Branges-Rovnyak space}
Suppose $B \in \HS(\HD, \C)$ is a Schur function, where $\HD$ is a separable Hilbert space and $\C$ is the complex field. Let $\T$ be the unit circle. If $\log (1- BB^*) \in L^1(\T)$, then there is a unique outer function $a$ with $a(0) > 0$ such that
$$|a(z)|^2 + B(z)B(z)^* = 1, a.e. ~z \in \T.$$
We call $a$ the {\it mate} of $B$. It is clear that $a$ is in $H^\infty_1$, the unit ball of $H^\infty$. We denote $\HM(T_a)$ and $\HM(T_{\overline{a}})$ by $\HM(a)$ and $\HM(\overline{a})$, respectively.

Suppose $B \in \HS(\HD, \C)$. If $\HD$ is finite dimensional, and $\HH(B)$ is $M_z$-invariant, then it was shown in \cite[Theorem 5.2]{AM19} that $\log (1- BB^*) \in L^1(\T)$. But if $\log (1- BB^*) \in L^1(\T)$, we don't necessarily have that $\HD$ is finite dimensional.
\begin{example}\label{infiniterankex}
Let $b_1(z) = \frac{1+z}{2\sqrt{2}}, b_i(z) = \frac{z^i}{\sqrt{2^i}}, i = 2, 3,\ldots$ and $B = (b_1, b_2, \ldots)$. Then
$$B(z)B(z)^* = \frac{|1+z|^2}{8} + \sum_{i=2}^\infty \frac{|z|^{2i}}{2^i} = \frac{|1+z|^2}{8} + \frac{1}{2}, z\in \T.$$
So
$$1-B(z)B(z)^* = \frac{1}{2} - \frac{|1+z|^2}{8}, z\in \T,$$
and $\log (1- BB^*) \in L^1(\T)$. In this case, we have the rank of $B$ is infinite, that is $\dim \HD = \infty$.
\end{example}
The following result is a main theorem in this section.
\begin{theorem}\label{matematrix}
Let $B \in \HS(\HD, \C)$ be such that $\log (1- BB^*) \in L^1(\T)$, and $a$ the mate of $B$. Then there exists an analytic outer function $A \in \HS(\HD, \HD)$ such that
\begin{align}\label{schurouter}
B(z)^*B(z) + A(z)^*A(z) = I_\HD, a.e. ~z \in \T.
\end{align}
Furthermore, there is an isometric embedding $J: \HH(B) \rightarrow H^2 \oplus H^2(\HD)$ satisfying the following properties.
\begin{itemize}
\item[(i)] A function $f \in H^2$ is contained in $\HH(B)$ if and only if there exists a unique $f^+ \in H^2(\HD)$ such that
$$B^* f + A^* f^+ \in \overline{H^2_0(\HD)}.$$
If this is the case, then $Jf = (f, f^+)$.
\item[(ii)] If $Jf = (f, f^+)$, then $JLf = (Lf, Lf^+)$. Consequently, $\|L^nf\|_{\HH(B)} \rightarrow 0$ as $n \rightarrow \infty$.
\item[(iii)] The orthogonal complement of $J\HH(B)$ in $H^2 \oplus H^2(\HD)$ is
$$[J\HH(B)]^\perp = \{(Bh, Ah): h \in H^2(\HD)\}.$$
\end{itemize}
\end{theorem}
\begin{proof}
Since $|a(z)|^2 + B(z)B(z)^* = 1$, a.e. $z \in \T$, we have $1 - B(z)B(z)^* > 0$, a.e. $z \in \T$. So $1 - B(z)B(z)^*$ is invertible for a.e. $z \in \T$. Thus for a.e. $z \in \T$,
$$I + B(z)^*(1 - B(z)B(z)^*)^{-1}B(z)$$
is an inverse of $I - B(z)^*B(z)$. Let $B(\HD)$ be the space of bounded linear operators on $\HD$. Then
$$\|(I - B(z)^*B(z))^{-1}\|_{B(\HD)} \leq 1 + (1 - B(z)B(z)^*)^{-1} \leq 2 (1 - B(z)B(z)^*)^{-1},$$
It follows that
$$\log^+ \|(I - B(z)^*B(z))^{-1}\|_{B(\HD)} \leq \log 2 (1 - B(z)B(z)^*)^{-1}$$
and
$$\log^+ \|(I - B(z)^*B(z))^{-1}\|_{B(\HD)} \in L^1(\T).$$
It is clear that $\log^+ \|I - B(z)^*B(z)\|_{B(\HD)} = 0$. Thus by \cite[Theorem 6.14]{RR97}, there exists an analytic outer function $A \in \HS(\HD, \HD)$ such that
$$B(z)^*B(z) + A(z)^*A(z) = I_\HD, a.e. ~z \in \T.$$

Now we use the idea in \cite{AM19} to prove (i)-(iii). Let $\HK = H^2 \oplus H^2(\HD)$, $U = \{(Bh, Ah): h \in H^2(\HD)\}$. Then $U$ is closed in $\HK$. Let $P$ be the orthogonal projection from $\HK$ to $H^2$, that is $P(f,g) = f, (f,g) \in \HK$.
We claim that $P$ is one to one on $\HK \ominus U$.
In fact, if $(0,g) \perp U$, then $g \perp Ah, h \in H^2(\HD)$. Since $A$ is outer we conclude that $g = 0$.

Let $\HH_0 = P(\HK \ominus U)$ with norm
$$\|f\|_{\HH_0}^2 = \|f\|_{H^2}^2 + \|g\|_{H^2(\HD)}^2,$$
where $P(f,g) = f$. We next show that $\HH_0 = \HH(B)$, and it is enough to show that they have the same reproducing kernel. Let
$$K_\lambda(z) = \left(\frac{1-B(z)B(\lambda)^*}{1-\overline{\lambda}z}, \frac{-A(z)B(\lambda)^*}{1-\overline{\lambda}z}\right) \in \HK.$$
Then for any $(Bh, Ah) \in U$, it is not hard to verify that
\begin{align*}
\langle K_\lambda, (Bh, Ah)\rangle_\HK = 0.
\end{align*}
Thus $K_\lambda \perp U$, and
$$PK_\lambda(z) = \frac{1-B(z)B(\lambda)^*}{1-\overline{\lambda}z} \in \HH_0.$$
Let $f \in \HH_0$. Then there is a unique $g \in H^2(\HD)$ such that $P(f,g) = f$. Then $(f,g) \perp (Bh, Ah), h \in H^2(\HD)$. So $B^*f + A^* g \in \overline{H^2_0(\HD)}$. It then follows that
$$\langle f, PK_\lambda\rangle_{\HH_0} = \langle (f,g), K_\lambda\rangle_{\HK} = f(\lambda).$$
Therefore $PK_\lambda$ is the reproducing kernel of $\HH_0$ and $\HH_0 = \HH(B)$. This proves (iii).

Let $J = P^{-1}: \HH(B) = \HH_0 \rightarrow \HK \ominus U \subseteq \HK$ be $Jf = (f, f^+)$. Then $J$ is an isometry, and $J\HH(B) = \HK \ominus U$. If $f \in \HH(B)$, then there is a unique $f^+ \in H^2(\HD)$ such that $(f, f^+) \in \HK \ominus U$. So $B^* f + A^* f^+ \in \overline{H^2_0(\HD)}$. This proves (i).

Suppose $Jf = (f, f^+)$ and $JLf = (Lf, g)$. Note that
$$B^*Lf + A^*Lf^+ = \overline{z} B^*f - \overline{z} B^*(f(0)) + \overline{z}A^*f^+ - \overline{z} A^* (f^+(0)) \in \overline{H^2_0(\HD)}.$$
Thus $g = Lf^+$. It then follows that
$$JL^nf = (L^nf, L^nf^+) \rightarrow 0.$$
This proves (ii).
\end{proof}

\begin{prop}\label{shiftinv}
Let $B \in \HS(\HD, \C)$ be such that $\log (1- BB^*) \in L^1(\T)$. Suppose $A \in \HS(\HD, \HD)$ is the outer function in Theorem \ref{matematrix} satisfying (\ref{schurouter}). If $A(0)$ is invertible, then $\HH(B)$ is $M_z$-invariant.
\end{prop}
\begin{proof}
Let $x \in \HD$. Then
\begin{align*}
& B(z)^*B(z)x + A(z)^*(-(A(0)^*)^{-1}x + A(z)x)\\
& = x - A(z)^* (A(0)^*)^{-1}x \in \overline{H^2_0(\HD)}.
\end{align*}
Thus $B(z) x \in \HH(B)$.

Let $\tau: \HD \rightarrow \HH(B)$ be defined by $\tau (u) = LB u, u \in \HD$.
Then $\tau$ is bounded and by \cite[Theorem 1]{BallBoloBasics}, we have
$$L^* f(z) = zf(z) - B(z)\tau^*(f), f \in \HH(B).$$
Therefore $zf \in \HH(B)$ whenever $f \in \HH(B)$ and so the conclusion holds.
\end{proof}
If $B \in \HS(\HD, \C)$ is such that $\HH(B)$ is $M_z$-invariant, in general we may not have $\log (1- BB^*) \in L^1(\T)$, see e.g. \cite[Example 4.4]{LGR}.
The following result is contained in \cite[Theorem 5.2]{AM19}.
\begin{corollary}[\cite{AM19}]\label{invariantmz}
Let $B \in \HS(\HD, \C)$ be such that $\log (1- BB^*) \in L^1(\T)$. If $\HD$ is finite dimensional, then $\HH(B)$ is $M_z$-invariant.
\end{corollary}
\begin{proof}
By Theorem \ref{matematrix}, there is an outer function $A \in \HS(\HD, \HD)$ such that
\begin{align*}
B(z)^*B(z) + A(z)^*A(z) = I_\HD, a.e. ~z \in \T.
\end{align*}
Since $\HD$ is finite dimensional, we have $A(0)$ is invertible. Thus the conclusion follows from Proposition \ref{shiftinv}.
\end{proof}

\subsection{Polynomial density}
When $B \in \HS(\HD, \C)$ has a mate, we show that the $\HH(B)$ space has many similar properties as the sub-Hardy space $\HH(b)$, where $b \in H^\infty_1$ has a mate, and we can use some similar ideas in \cite{Sa86}.
\begin{lemma}\label{containment}
Let $B \in \HS(\HD, \C)$ be such that $\log (1- BB^*) \in L^1(\T)$, and $a$ the mate of $B$. Then
$$\HM(a) \subseteq \HM(\overline{a}) \subseteq \HH(B),$$
and the inclusions are contractive. Consequently, $\HH(B)$ contains all the polynomials.
\end{lemma}
\begin{proof}
By Douglas lemma (\cite{Do66}, see also \cite[Theorem 16.7]{FM162}), $\HM(a) \subseteq \HM(\overline{a})$ contractively if and only if
$$T_a T_{\overline{a}} \leq T_{\overline{a}} T_a,$$
which is true since $T_a$ is subnormal.

Also by Douglas lemma, $\HM(\overline{a}) \subseteq \HH(B)$ contractively if and only if
$$T_{\overline{a}} T_a \leq I - T_B T_B^*.$$
Note that
$$T_{\overline{a}} T_a = T_{|a|^2} = T_{1-BB^*} = I - T_{BB^*},$$
and it is not hard to check that
$$T_B T_B^* \leq T_{BB^*}.$$
We conclude that $T_{\overline{a}} T_a \leq I - T_B T_B^*$ and $\HM(\overline{a}) \subseteq \HH(B)$ contractively.

It is known that if $a$ is an outer function, then $T_{\overline{a}}$ is one to one, and $\ran T_{\overline{a}}$ contains all the polynomials. Thus $\HH(B)$ contains all the polynomials.
\end{proof}
If $\varphi \in H^\infty$, then $\varphi$ is also a multiplier of $H^2(\HD)$ with norm $\|\varphi\|_\infty$.
\begin{lemma}\label{contractionhbs}
Let $B \in \HS(\HD, \C)$ and $\varphi \in H^\infty_1$. Then $T_{\overline{\varphi}}$ is a contractive operator from $\HH(B^*)$ to $\HH(B^*)$.
\end{lemma}
\begin{proof}
Note that $\HH(B^*)$ is the range space of $(I-T_B^*T_B)^{1/2}$. Thus $T_{\overline{\varphi}}: \HH(B^*) \rightarrow \HH(B^*)$ is a contraction if and only if
$$T_{\overline{\varphi}} (I-T_B^*T_B) T_\varphi \leq I-T_B^*T_B,$$
see e.g. \cite[Corollary 16.10]{FM162}.
By
\begin{align*}
I-T_B^*T_B - T_{\overline{\varphi}} (I-T_B^*T_B) T_\varphi & = I - T_{|\varphi|^2} - T_B^* (I - T_{|\varphi|^2}) T_B\\
& = T_{(I-B^*B)(1-|\varphi|^2)} \geq 0,
\end{align*}
we conclude that $T_{\overline{\varphi}}: \HH(B^*) \rightarrow \HH(B^*)$ is a contraction.
\end{proof}
The following result is known.
\begin{lemma}[\cite{Sa94}]\label{relationhb}
Suppose $A$ is a contraction from a Hilbert space $\HH$ to a Hilbert space $\HK$. Then $f \in \HH(A)$ if and only if $A^*f \in \HH(A^*)$. Furthermore,
$$\|f\|_{\HH(A)}^2 = \|f\|_{\HK}^2 + \|A^*f\|_{\HH(A^*)}^2, f \in \HH(A).$$
\end{lemma}

\begin{lemma}\label{contractionhbs2}
Let $B \in \HS(\HD, \C)$ and $\varphi \in H^\infty_1$. Then $T_{\overline{\varphi}}$ is a contractive operator from $\HH(B)$ to $\HH(B)$.
\end{lemma}
\begin{proof}
We present two proofs here.
Method one: If $f \in \HH(B)$, then $T_B^* f\in \HH(B^*)$. By Lemma \ref{contractionhbs},
$$T_B^* T_{\overline{\varphi}} f = T_{\overline{\varphi}} T_B^* f \in \HH(B^*).$$
Thus $T_{\overline{\varphi}} f \in \HH(B)$. Note that
\begin{align*}
\|T_{\overline{\varphi}} f\|_{\HH(B)}^2& = \|T_{\overline{\varphi}} f\|_{H^2}^2 + \|T_B^*T_{\overline{\varphi}} f\|_{\HH(B^*)}^2\\
& \leq \| f\|_{H^2}^2 + \|T_B^* f\|_{\HH(B^*)}^2 = \| f\|_{\HH(B)}^2.
\end{align*}
So $T_{\overline{\varphi}}: \HH(B) \rightarrow \HH(B)$ is a contraction.

Method two: Recall that $\HH(B) \subseteq H^2$ and $T_{\overline{z}} = L: \HH(B) \rightarrow \HH(B)$ is a contraction (\cite{BallBoloBasics}). Then by von Neumann's inequality, for any polynomial $p$ we have
$$\|p(L)\|_{B(\HH(B))} \leq \|p\|_{H^\infty}.$$
Thus $\|T_{\overline{p}}\|_{B(\HH(B))} \leq \|p\|_{H^\infty}$. Let $\varphi \in H^\infty_1$. Then there are polynomials $p_n$ such that $\|p_n\|_\infty \leq \|\varphi\|_\infty$ and $p_n(z) \rightarrow \varphi(z), z \in \D$ as $n \rightarrow \infty$. It follows that $T_{\overline{p_n}} \rightarrow T_{\overline{\varphi}}$ in the weak operator topology of $B(H^2)$. So for any $f \in \HH(B)$, $T_{\overline{p_n}} f \rightarrow T_{\overline{\varphi}} f$ weakly in $H^2$. Since
$$\|T_{\overline{p_n}}\|_{B(\HH(B))} \leq \|p_n\|_{H^\infty} \leq \|\varphi\|_\infty,$$
we also have $T_{\overline{p_n}} f$ converges weakly in $\HH(B)$. Therefore $T_{\overline{p_n}} f \rightarrow T_{\overline{\varphi}} f$ weakly in $\HH(B)$. Hence $T_{\overline{\varphi}} f \in \HH(B)$ and $\|T_{\overline{\varphi}} f\|_{B(\HH(B))} \leq \| f\|_{\HH(B)}$.
\end{proof}

\begin{lemma}\label{mathbackinv}
Let $B \in \HS(\HD, \C)$ be such that $\log (1- BB^*) \in L^1(\T)$. Suppose $A \in \HS(\HD, \HD)$ is the outer function in Theorem \ref{matematrix} satisfying (\ref{schurouter}). Let $\varphi \in H^\infty$, then $(T_{\overline{\varphi}}f)^+  = T_{\overline{\varphi}}f^+$ and
$$\|T_{\overline{\varphi}}f\|_{\HH(B)}^2 = \|T_{\overline{\varphi}}f\|_{H^2}^2 + \|T_{\overline{\varphi}}f^+\|_{H^2(\HD)}^2, f \in \HH(B),$$
where $f^+ \in H^2(\HD)$ is the unique function such that $B^*f + A^* f^+ = 0$.
\end{lemma}
\begin{proof}
By Lemma \ref{contractionhbs2}, $T_{\overline{\varphi}}f \in \HH(B)$. Then by Theorem \ref{matematrix}, there is $(T_{\overline{\varphi}}f)^+ \in H^2(\HD)$ such that
$$T_B^* T_{\overline{\varphi}}f + T_A^* (T_{\overline{\varphi}}f)^+ = 0.$$
Since $T_B^* f + T_A^* f^+ = 0$, we have
$$T_B^* T_{\overline{\varphi}}f + T_A^* T_{\overline{\varphi}}f^+ = T_{\overline{\varphi}} T_B^* f + T_{\overline{\varphi}} T_A^* f^+ = 0.$$
Thus $(T_{\overline{\varphi}}f)^+  = T_{\overline{\varphi}}f^+$. The norm identity for $T_{\overline{\varphi}}f$ then follows from Theorem \ref{matematrix}.
\end{proof}

In the following, when we consider $L$-invariant subspaces, they are assumed to be closed.
\begin{lemma}\label{backwarddense}
Let $B \in \HS(\HD, \C)$ be such that $\log (1- BB^*) \in L^1(\T)$, and $a$ the mate of $B$. Let $\HN$ be an $L$-invariant subspace of $\HH(B)$. Then $T_{\overline{a}}\HN$ is dense in $\HN$.
\end{lemma}
\begin{proof}
By the proof of method two in Lemma \ref{contractionhbs2}, we have $\HN$ is $T_{\overline{a}}$-invariant, and $\|T_{\overline{a}}\|_{B(\HH(B))} \leq \|a\|_\infty$. Suppose $f \in \HN$ and $f \perp T_{\overline{a}}\HN$. We need to show that $f = 0$.

For $n \geq 0$, we have $L^n f \in \HN$. So $f \perp T_{\overline{a}} L^n f$. Note that $T_{\overline{a}} L^n f = T_{\overline{az^n}} f$ and
\begin{align*}
0 = \langle f, T_{\overline{az^n}}f\rangle_{\HH(B)}& = \langle f, T_{\overline{az^n}} f\rangle_{H^2} + \langle f^+, T_{\overline{az^n}} f^+\rangle_{H^2(\HD)}\\
& = \int_\T a(z)z^n |f(z)|^2 \frac{|dz|}{2\pi} + \int_\T a(z)z^n \|f^+(z)\|_{\HD}^2 \frac{|dz|}{2\pi}\\
& = \int_\T a(z)z^n (|f(z)|^2 + \|f^+(z)\|_{\HD}^2) \frac{|dz|}{2\pi}.
\end{align*}
Let $h(z) = a(z)(|f(z)|^2 + \|f^+(z)\|_{\HD}^2)$. Then $h \in L^1(\T)$ and $\widehat{h}(n) = 0, n \leq 0$. Thus $h \in H_0^1$. Since $a$ is outer, we obtain that $|f(z)|^2 + \|f^+(z)\|_{\HD}^2 \in H_0^1$. It follows that $|f(z)|^2 + \|f^+(z)\|_{\HD}^2 = 0$ and hence $f = 0$.
\end{proof}
The above lemma can also be deduced from \cite[Lemma 4.2]{FMR20}.
When $B$ has finite rank, that is, $B \in \HS(\HD, \C)$ with $\dim \HD < \infty$, the following result was proved in \cite[Theorem 5.5]{AM19}. We restate Theorem A in the following.
\begin{theorem}\label{polydense}
Let $B \in \HS(\HD, \C)$ be such that $\log (1- BB^*) \in L^1(\T)$, and $a$ the mate of $B$. Then the polynomials are dense in $\HH(B)$.
\end{theorem}
\begin{proof}
Lemma \ref{backwarddense} implies that $T_{\overline{a}}\HH(B)$ is dense in $\HH(B)$. By $\HH(B) \subseteq H^2$, we have $\HM(\overline{a})$ is dense in $\HH(B)$. Since $T_{\overline{a}} z^n, n \geq 0$ span $\HM(\overline{a})$, and $T_{\overline{a}} z^n$ are polynomials, we conclude from Lemma \ref{containment} that the polynomials are dense in $\HH(B)$.
\end{proof}

\subsection{Backward shift invariant subspaces}
Suppose $B \in \HS(\HD, \C)$ is a Schur function. When $\dim \HD = 1$ or $\dim \HD < +\infty$, the following result was obtained in \cite[Theorem 5]{Sa86} or \cite[Theorem 5.12]{AM19}. We restate Theorem B in the following.
\begin{theorem}\label{backwardinv}
Let $B \in \HS(\HD, \C)$ be such that $\log (1- BB^*) \in L^1(\T)$, and $a$ the mate of $B$. Then any proper $L$-invariant subspace of $\HH(B)$ is of the form
$$\HH(B) \cap K_\theta,$$
where $\theta$ is an inner function and $K_\theta = H^2 \ominus \theta H^2$.
\end{theorem}
\begin{proof}
Let $\HN$ be a proper $L$-invariant subspace of $\HH(B)$. Let $\HE = \text{clos}_{H^2} \HN$, the closure of $\HN$ in $H^2$. We claim that $\HN = \HE \cap \HH(B)$.

It is clear that $\HN \subseteq \HE \cap \HH(B)$. To prove the reverse inclusion, we first show that $T_{\overline{a}} \HE$ is contained in $\HN$. Let $f \in \HE$. Then there exist $f_n \in \HN$ such that $f_n \rightarrow f \in H^2$. So $T_{\overline{a}} f_n \rightarrow T_{\overline{a}} f$ in the norm of $\HM(\overline{a})$. By Lemma \ref{containment}, we have $T_{\overline{a}} f_n \rightarrow T_{\overline{a}} f$ in $\HH(B)$. Since $\HN$ is $L$-invariant, it follows form Lemma \ref{backwarddense} that $T_{\overline{a}} f_n \in \HN$ and so $T_{\overline{a}} f \in \HN$. Thus $T_{\overline{a}} \HE \subseteq \HN$, and $T_{\overline{a}}(\HE \cap \HH(B)) \subseteq \HN$. Note that $\HE \cap \HH(B)$ is closed in $\HH(B)$ and $L$-invariant. Hence Lemma \ref{backwarddense} ensures that $\HE \cap \HH(B)$ is $T_{\overline{a}}$-invariant and $T_{\overline{a}}(\HE \cap \HH(B))$ is dense in $\HE \cap \HH(B)$. Therefore $\HE \cap \HH(B) \subseteq \HN$ and so $\HN = \HE \cap \HH(B)$. The conclusion of the theorem then follows from Beurling's theorem.
\end{proof}
The above theorem also follows from Theorem 4.1 or Theorem 4.6 in \cite{FMR20}.

\section{Carath\'{e}odory condition and Cyclic vectors}
\subsection{Carath\'{e}odory condition}
Let $B \in \HS(\HD, \C)$. If $\xi = (\xi_1, \xi_2, \ldots) \in l^2, \|\xi\|_{l^2} \leq 1$, then $\frac{1+B(z)\xi^*}{1-B(z)\xi^*}, z \in \D$ has positive real part, where we think of $B(z) = (b_1(z), b_2(z), \ldots)$, $z \in \D$ as a vector in $l^2$ and $B(z)\xi^* = \sum_{i=1}^\infty b_i(z) \overline{\xi_i}$ denotes the inner product between $B(z)$ and $\xi$ in $l^2$.  So there is a unique positive finite Borel measure $\mu_\xi$ on $\T$ such that
$$H_\xi(z): = \frac{1+B(z)\xi^*}{1-B(z)\xi^*} = \int_\T \frac{\zeta+z}{\zeta-z} d\mu_\xi(\zeta) + i \text{Im} \frac{1+B(0)\xi^*}{1-B(0)\xi^*}, z \in \D.$$
We call $\mu_\xi$ the Aleksandrov-Clark measure for $B$ associated with $\xi$. We then have
$$\frac{1-|B(z)\xi^*|^2}{|1-B(z)\xi^*|^2} = \int_\T \frac{1-|z|^2}{|\zeta-z|^2} d\mu_\xi(\zeta),$$
and
\begin{align}\label{aclarkforB}
(1-B(z)\xi^*)^{-1} = \int_\T \frac{1}{1-z\overline{\zeta}}d\mu_\xi(\zeta) + \frac{1-\overline{H_\xi(0)}}{2}.
\end{align}
The proof of the following lemma uses the same argument as in \cite[Theorem 3.1]{LMT}.
\begin{lemma}\label{limitsclark}
Let $B \in \HS(\HD, \C)$, $\xi \in l^2, \|\xi\|_{l^2} \leq 1$. Then for any $\lambda \in \T$, the non-tangential limit
$$\mu_\xi(\lambda) = \angle\lim_{z \rightarrow \lambda} (1-z\overline{\lambda})(1-B(z)\xi^*)^{-1}$$
exists.
\end{lemma}
\begin{proof}
By (\ref{aclarkforB}), we have
\begin{align*}
\angle\lim_{z \rightarrow \lambda} (1-z\overline{\lambda})(1-B(z)\xi^*)^{-1}& = \angle\lim_{z \rightarrow \lambda} (1-z\overline{\lambda}) \int_\T \frac{1}{1-z\overline{\zeta}}d\mu_\xi(\zeta)\\
& = \mu_\xi(\lambda) + \angle\lim_{z \rightarrow \lambda} \int_{\T\backslash \{\lambda\}} \frac{1-z\overline{\lambda}}{1-z\overline{\zeta}}d\mu_\xi(\zeta)
\end{align*}
Let $g_z(\zeta) = \frac{1-z\overline{\lambda}}{1-z\overline{\zeta}}, \zeta \in \T\backslash \{\lambda\}$. Note that in the non-tangential region
$$\Gamma_\alpha(\lambda) = \{z \in \D: |z-\lambda| < \alpha (1-|z|)\}, \alpha >1$$
$$|g_z(\zeta)| = \frac{|z - \lambda|}{|z-\zeta|} < \alpha \frac{1-|z|}{|z-\zeta|} \leq \alpha.$$
Since $\angle\lim_{z \rightarrow \lambda} g_z(\zeta) = 0, \zeta \in \T\backslash \{\lambda\}$, we conclude from Lebesgue dominated convergence theorem that
$\angle\lim_{z \rightarrow \lambda} (1-z\overline{\lambda})(1-B(z)\xi^*)^{-1} = \mu_\xi(\lambda)$.
\end{proof}

Let $B \in \HS(\HD, \C), \lambda \in\T$. If there is $\xi \in l^2, \|\xi\|_{l^2} = 1$ such that $\angle\lim_{z \rightarrow \lambda} (1-B(z)\xi^*)(1-z\overline{\lambda})^{-1}$ exists, then we say that $B$ satisfies {\it Carath\'{e}odory condition} at $\lambda$. By Lemma \ref{limitsclark}, $B$ satisfies Carath\'{e}odory condition at $\lambda$ if and only if $\mu_\xi(\lambda) \neq 0$. When $B$ is a scalar-valued Schur function $b$, the following result was obtained in \cite{Sa88}.
\begin{theorem}\label{non-tangentiallimit}
Let $B \in \HS(\HD, \C), \lambda \in\T$. Then the following are equivalent.
\end{theorem}
\begin{itemize}
\item[(i)] $$\liminf_{z \rightarrow \lambda} \frac{1-\|B(z)\|_{l^2}}{1-|z|} = c < +\infty.$$
\item[(ii)] There is $\xi \in l^2, \|\xi\|_{l^2} = 1$ such that
$$\frac{1-B(z)\xi^*}{1-z\overline{\lambda}} \in \HH(B).$$
\item[(iii)] For any $f \in \HH(B)$, $f$ has a non-tangential limit at $\lambda$.
\item[(iv)] $B$ satisfies Carath\'{e}odory condition at $\lambda$, i.e. there is $\xi \in l^2, \|\xi\|_{l^2} = 1$ such that $$\angle\lim_{z \rightarrow \lambda} (1-B(z)\xi^*)(1-z\overline{\lambda})^{-1}$$
    exists.
\item[(v)] There is $\xi \in l^2, \|\xi\|_{l^2} = 1$ such that $\mu_\xi(\lambda) \neq 0$, where $\mu_\xi$ is the Aleksandrov-Clark measure for $B$ associated with $\xi$.
\end{itemize}
\begin{proof}
(i) $\Rightarrow$ (ii) If $\liminf_{z \rightarrow \lambda} \frac{1-\|B(z)\|_{l^2}}{1-|z|} = c < +\infty$, then there is $z_n \rightarrow \lambda$ such that $K_{z_n}^B$ is bounded in $\HH(B)$, so there is a subsequence still denoted by $\{z_n\}$ such that $K_{z_n}^B \rightarrow K \in \HH(B)$ weakly. Since $B \in \HS(\HD, \C)$, $B(z_n)$ is also bounded in $l^2$, and there is a subsequence denoted by $\{z_n\}$ again such that $B(z_n) \rightarrow \xi$ weakly in $l^2$. Then
\begin{align*}
K(z)&  = \langle K, K^B_z\rangle_{\HH(B)} = \lim_{n\rightarrow \infty}  \langle K_{z_n}, K^B_z\rangle_{\HH(B)}\\
& = \lim_{n\rightarrow \infty} K_{z_n}^B(z) = \lim_{n\rightarrow \infty} \frac{1-B(z)B(z_n)^*}{1-z\overline{z_n}}\\
& = \frac{1-B(z)\xi^*}{1-z\overline{\lambda}}, z \in \D.
\end{align*}
Thus $\|\xi\|_{l^2} = 1$ and (ii) holds.

(ii) $\Rightarrow$ (iii) Let $K(z) = \frac{1-B(z)\xi^*}{1-z\overline{\lambda}} \in \HH(B)$. We first prove that $K_z^B$ is bounded in any non-tangential region. For $z \in \Gamma_\alpha(\lambda) = \{z \in \D: |z-\lambda| < \alpha (1-|z|)\}$,
\begin{align*}
|K(z)|& = \frac{|1-B(z)\xi^*|}{|1-z\overline{\lambda}|} \geq \frac{1-\|B(z)\|_{l^2}}{|z-\lambda|}\\
& = \frac{1-\|B(z)\|_{l^2}^2}{1-|z|^2}\frac{1-|z|^2}{1+\|B(z)\|_{l^2}}\frac{1}{|z-\lambda|}\\
& = \|K_z^B\|_{\HH(B)}^2 \frac{1+|z|}{1+\|B(z)\|_{l^2}}\frac{1-|z|}{|z-\lambda|}\\
& > \frac{1}{\alpha} \|K_z^B\|_{\HH(B)}^2 \frac{1+|z|}{1+\|B(z)\|_{l^2}}.
\end{align*}
It then follows from $|K(z)| \leq \|K\|_{\HH(B)} \|K_z^B\|_{\HH(B)}$ that $\|K_z^B\|_{\HH(B)} < 2\alpha \|K\|_{\HH(B)}$ for $z \in \Gamma_\alpha(\lambda)$. This also implies (i) holds.

Now we show for any $f \in \HH(B)$, $f$ has a non-tangential limit at $\lambda$. By $K \in \HH(B) \subseteq H^2$, we obtain that $|K(z)| \precsim (1-|z|)^{-1/2}$, and
$$1-B(z)\xi^* = (1-z\overline{\lambda}) K(z) \rightarrow 0$$
as $z \rightarrow \lambda$ non-tangentially. We claim that $B(z) \rightarrow \xi$ weakly in $l^2$ as $z \rightarrow \lambda$ non-tangentially. In fact, let $\{z_n\}$ be any sequence tending to $\lambda$ non-tangentially and such that $B(z_n) \rightarrow \xi_0$ weakly in $l^2$. Then $\|\xi_0\|_{l^2} \leq 1$. Since $1 - B(z) \xi^* \rightarrow 0$ as $z \rightarrow \lambda$ non-tangentially, we have $\xi_0 \xi^* = 1$. Since $\|\xi\|_{l^2} = 1$ and $\|\xi_0\|_{l^2} \leq 1$, it implies that $\xi_0 = \xi$. Thus $B(z) \rightarrow \xi$ weakly in $l^2$ as $z \rightarrow \lambda$ non-tangentially.
If we denote $\xi$ by $B(\lambda)$, then $K(z) = \frac{1-B(z)B(\lambda)^*}{1-z\overline{\lambda}}$, and
$$K^B_z(w) = \frac{1-B(w)B(z)^*}{1-w\overline{z}} \rightarrow K(w), w \in \D$$
as $z \rightarrow \lambda$ non-tangentially. So $\angle\lim_{z \rightarrow \lambda} \langle K_z^B, K_w^B\rangle_{\HH(B)} = \langle K, K_w^B\rangle_{\HH(B)}$. Since the reproducing kernels are dense in $\HH(B)$ and $\|K_z^B\|_{\HH(B)} < 2\alpha \|K\|_{\HH(B)}$ for $z \in \Gamma_\alpha(\lambda)$, we conclude that $K_{z}^B \rightarrow K \in \HH(B)$ weakly as $z \rightarrow \lambda$ non-tangentially. So (iii) holds, and we can think of $K$ as the reproducing kernel $K_\lambda^B$ at $\lambda$.

(iii) $\Rightarrow$ (i) If for any $f \in \HH(B)$, $f$ has a non-tangential limit at $\lambda$, then
$$f(\lambda) = \angle\lim_{z \rightarrow \lambda} \langle f, K_z^B\rangle_{\HH(B)}.$$
Thus by the uniform boundedness principle,
$$\sup_{z \in \Gamma_\alpha(\lambda)} \|K_z^B\|_{\HH(B)} < + \infty.$$
So (i) holds.

It is clear that (ii) and (iii) imply (iv). By Lemma \ref{limitsclark}, we see that (iv) is equivalent to (v). Now we show (iv) implies (i). Let $z = r\lambda, 0 < r <1$. Then (i) follows from the following inequality,
\begin{align*}
\left|\frac{1-B(r\lambda)\xi^*}{1-r\lambda\overline{\lambda}}\right| \geq \frac{1-\|B(r\lambda)\|_{l^2}}{1-r}.
\end{align*}
\end{proof}

Let $b(z) = B(z)\xi^*$. Then $b \in \HS(\C,\C)$, and one checks that $\HH(B)$ is contractively contained in $\HH(b)$. So we can also use the result in \cite{Sa88} to derive some of the equivalences in the above theorem.

If $B\in \HS(\HD,\C)$ satisfies Carath\'{e}odory condition at $\lambda$, then $B$ has non-tangential limit at $\lambda$ and we let $K_\lambda^B(z) = \frac{1-B(z)B(\lambda)^*}{1-z\overline{\lambda}}$.
\begin{corollary}\label{normconvergence}
Suppose $B \in \HS(\HD,\C)$ satisfies Carath\'{e}odory condition at $\lambda$. Then $K_z^B \rightarrow K_\lambda^B$ in $\HH(B)$ as $z \rightarrow \lambda$ non-tangentially. Consequently, if $B \in \HS(\HD,\C)$ satisfies Carath\'{e}odory condition at $\lambda$, then
$$\liminf_{z \rightarrow \lambda} \frac{1-\|B(z)\|_{l^2}}{1-|z|} = \angle\lim_{z \rightarrow \lambda} \frac{1-B(z)B(\lambda)^*}{1-z\overline{\lambda}} = \frac{1}{\mu_{B(\lambda)}(\lambda)},$$
where $\mu_{B(\lambda)}$ is the Aleksandrov-Clark measure for $B$ associated with $B(\lambda)$.
\end{corollary}
\begin{proof}
We know from Theorem \ref{non-tangentiallimit} that $K_{z}^B \rightarrow K_\lambda^B$ weakly as $z \rightarrow \lambda$ non-tangentially. So it is enough to show that $\|K_z^B\|_{\HH(B)} \rightarrow \|K_\lambda^B\|_{\HH(B)}$ as $z \rightarrow \lambda$ non-tangentially.

Note that
$$\angle\lim_{z \rightarrow \lambda} \frac{1-B(z)B(\lambda)^*}{1-z\overline{\lambda}} = \angle\lim_{z \rightarrow \lambda} \langle K_\lambda^B, K_z^B\rangle_{\HH(B)} = \|K_\lambda^B\|_{\HH(B)}^2.$$
Let $g(z) = K_\lambda^B(z) - \|K_\lambda^B\|_{\HH(B)}^2$. Then
$$B(z)B(\lambda)^* = 1 - (1-z\overline{\lambda})(g(z) + \|K_\lambda^B\|_{\HH(B)}^2).$$
So
\begin{align*}
&|B(z)B(\lambda)^*|^2 \\
&= 1 + |1-z\overline{\lambda}|^2 \left|g(z) + \|K_\lambda^B\|_{\HH(B)}^2\right|^2 -2\text{Re} (1-z\overline{\lambda})(g(z) + \|K_\lambda^B\|_{\HH(B)}^2).
\end{align*}
It is clear that
$$ \angle\lim_{z \rightarrow \lambda} \frac{|1-z\overline{\lambda}|^2 \left|g(z) + \|K_\lambda^B\|_{\HH(B)}^2\right|^2}{1-|z|^2} = 0.$$
Since $\angle\lim_{z \rightarrow \lambda} g(z) = 0$, we conclude that
$$\angle\lim_{z \rightarrow \lambda} \frac{1-|B(z)B(\lambda)^*|^2}{1-|z|^2} = \angle\lim_{z \rightarrow \lambda} \frac{2\text{Re} (1-z\overline{\lambda})\|K_\lambda^B\|_{\HH(B)}^2}{1-|z|^2}.$$
But since
$$\frac{2\text{Re} (1-z\overline{\lambda})}{1-|z|^2} = 1 + \frac{|z-\lambda|^2}{1-|z|^2},$$
and
$$\frac{1-|B(z)B(\lambda)^*|^2}{1-|z|^2} \geq \frac{1-\|B(z)\|_{l^2}^2}{1-|z|^2},$$
it follows that
\begin{align}\label{limsup}
\limsup \|K_z^B\|_{\HH(B)}^2 = \limsup \frac{1-\|B(z)\|_{l^2}^2}{1-|z|^2} \leq \|K_\lambda^B\|_{\HH(B)}^2
\end{align}
 as $z \rightarrow \lambda$ non-tangentially.
Since $K_{z}^B \rightarrow K_\lambda^B$ weakly as $z \rightarrow \lambda$ non-tangentially, we have
$$\|K_\lambda^B\|_{\HH(B)} \leq \liminf \|K_z^B\|_{\HH(B)}$$
as $z \rightarrow \lambda$ non-tangentially.
Thus $\|K_z^B\|_{\HH(B)} \rightarrow \|K_\lambda^B\|_{\HH(B)}$ as $z \rightarrow \lambda$ non-tangentially. The proof is complete.
\end{proof}
Once we have proved (\ref{limsup}), we can also use the following fact to obtain the conclusion in the above theorem: If $x_n \rightarrow x$ weakly in a Hilbert space and if $\limsup_{n \rightarrow \infty} \|x_n\|^2 \leq \|x\|^2$, then $\|x_n - x\|^2 \rightarrow 0$ as $n \rightarrow \infty$.
\begin{example}\label{infiniterankex2}
Let $b_1(z) = \frac{1+z}{2\sqrt{2}}, b_i(z) = \frac{z^i}{\sqrt{2^i}}, i = 2, 3,\ldots$ and $B = (b_1, b_2, \ldots)$. Then
$$B(1) = (\frac{1}{\sqrt{2}}, \frac{1}{\sqrt{2^2}}, \frac{1}{\sqrt{2^3}}, \ldots).$$
So
$$B(z)B(1)^* = \frac{1+z}{4} + \frac{z^2}{4-2z}, z\in \D,$$
and
$$\angle\lim_{z \rightarrow 1} (1-B(z)B(1)^*)(1-z)^{-1}$$
exists. Thus $B$ satisfies Carath\'{e}odory condition at $1$, and by Theorem \ref{non-tangentiallimit}, every $f \in \HH(B)$ has a non-tangential limit at $1$.
\end{example}

\subsection{Cyclic vectors} Let $B \in \HS(\HD, \C)$ be such that $\log (1- BB^*) \in L^1(\T)$. Then $B$ has a mate, and we have shown in Theorem \ref{polydense} that the polynomials are dense in $\HH(B)$. If furthermore, $A \in \HS(\HD, \HD)$ is the outer function in Theorem \ref{matematrix} and $A(0)$ is invertible, then $\HH(B)$ is $M_z$-invariant. In particular, when $\HD$ is finite dimensional, $\HH(B)$ is $M_z$-invariant.
\begin{lemma}\label{backforwshift}
Let $B \in \HS(\HD, \C)$ be such that $\HH(B)$ is $M_z$-invariant and polynomials are dense in $\HH(B)$. Suppose $B = (b_1, b_2, \ldots)$, then
$$L = M_z^* - \sum_{i=1}^\infty Lb_i \otimes b_i,$$
where the sum converges in the strong operator topology.
\end{lemma}
\begin{proof}
Let $\tau: \HD \rightarrow \HH(B)$ be defined by $\tau (u) = LB u, u \in \HD$.
Then $\tau$ is bounded and by \cite[Theorem 1]{BallBoloBasics}, we have
$$L^* f(z) = zf(z) - B(z)\tau^*(f), f \in \HH(B).$$
One checks that $B(z)\tau^*(f) = \sum_{i=1}^\infty \langle f, Lb_i\rangle_{\HH(B)} b_i(z)$. Note that each $b_i \in \HH(B)$. In fact, if we identify $\HD$ with $l^2$, let $\{e_i\}$ be the standard orthonormal basis of $l^2$, then $\tau (e_i) = LB e_i = Lb_i \in \HH(B)$. So $b_i = z Lb_i + b_i(0) \in \HH(B)$. Thus
$$L^* = M_z - \sum_{i=1}^\infty b_i \otimes Lb_i,$$
and the conclusion follows.
\end{proof}
In the setting of the above lemma, if furthermore, $\sum_{i=1}^\infty b_i \otimes Lb_i$ is compact, then the spectrum $\sigma(M_z) = \overline{\D}$. Since in this case, the essential spectra of $M_z$ and $L^*$ are the same. Also noting that $L$ is a contraction, so when $|\lambda| > 1$, $M_z - \lambda$ is Fredholm  with index zero, hence it is invertible.

When polynomials are dense in $\HH(B)$, for $\lambda \in \T$, let $e_\lambda: \HH(B) \rightarrow \C$ be densely defined by $e_\lambda(p) = p(\lambda)$, where $p$ is a polynomial.
\begin{lemma}\label{pointevl}
Let $B \in \HS(\HD, \C)$ be such that $\HH(B)$ is $M_z$-invariant and polynomials are dense in $\HH(B)$. Let $\lambda \in \T$. Then the following are equivalent.
\begin{itemize}
\item[(i)] $e_\lambda$ extends to be a bounded operator from $\HH(B)$ to $\C$.
\item[(ii)] $\overline{\lambda}$ is an eigenvalue of $M_z^*$.
\item[(iii)] $[z-\lambda] \neq \HH(B)$, where $[z-\lambda]$ is the closure of the polynomial multiples of $z-\lambda$.
\item[(iv)] Every function $f \in \HH(B)$ has a non-tangential limit at $\lambda$.
\end{itemize}
\end{lemma}
\begin{proof}
(i) $\Rightarrow$ (ii) If $e_\lambda$ extends to be a bounded operator from $\HH(B)$ to $\C$, then the extension is unique and we still denote this operator by $e_\lambda$. So there exists $K \in \HH(B), K \neq 0$ such that
$$f(\lambda): = e_\lambda(f) = \langle f, K\rangle_{\HH(B)}.$$
Then $e_\lambda(zf) = \lambda f(\lambda)$ and
$$\langle f, M_z^* K\rangle_{\HH(B)} = \langle zf, K\rangle_{\HH(B)} = \lambda f(\lambda) = \langle f, \overline{\lambda} K\rangle_{\HH(B)}.$$
So $\overline{\lambda}$ is an eigenvalue of $M_z^*$.

(ii) $\Rightarrow$ (iii) Suppose $M_z^* K = \overline{\lambda} K$ with $K \in \HH(B)$. Then it is clear that $(z-\lambda)p$ is orthogonal to $K$ for any polynomial $p$. Thus $[z-\lambda] \neq \HH(B)$.

(iii) $\Rightarrow$ (i) If $[z-\lambda] \neq \HH(B)$, since polynomials are dense in $\HH(B)$, we have $\HH(B) = [z - \lambda] + \C$ and $\HH(B)/ [z - \lambda]$ is one dimensional. Let $\pi: \HH(B) \rightarrow \HH(B)/ [z - \lambda]$ be the quotient map, and let $\varphi: \HH(B)/ [z - \lambda] \rightarrow \C$ be the linear map defined by $\varphi (\pi(1)) = 1$. Then $e_\lambda(p) = \varphi (\pi(p))$, $p$ polynomial. Thus $e_\lambda$ is bounded.

(iv) $\Rightarrow$ (i) If every function $f \in \HH(B)$ has a non-tangential limit at $\lambda$, then from the proof of Theorem \ref{non-tangentiallimit}, we know that $B$ has a non-tangential limit at $\lambda$, and $K_\lambda^B(z) = \frac{1-B(z)B(\lambda)^*}{1-z\overline{\lambda}} \in \HH(B)$ satisfying
$$f(\lambda) = \langle f, K_\lambda^B\rangle_{\HH(B)}.$$
Thus (i) holds.

(i) and (ii) $\Rightarrow$ (iv) Suppose $M_z^* K = \overline{\lambda} K$ with $f(\lambda) : = e_\lambda(f) = \langle f, K\rangle_{\HH(B)}, f \in \HH(B)$. Note that $\tau: l^2 \rightarrow \HH(B)$ defined by $\tau (u) = LB u, u \in l^2$ is bounded. Thus for $u = (u_i)_i \in l^2$, $LB u = \sum_{i=1}^\infty Lb_i u_i$ converges in norm in $\HH(B)$. So $Bu = \sum_{i=1}^\infty b_i u_i = z LB u + \sum_{i = 1}^\infty b_i(0) u_i$ converges in $\HH(B)$.
Then
\begin{align*}
\overline{K(0)}&  = \langle K_0^B, K\rangle_{\HH(B)}\\
& = K_0^B(\lambda) = e_\lambda (K_0^B) = e_\lambda \left(1 - \sum_{i=1}^\infty b_i \overline{b_i(0)}\right)\\
& = 1 - \sum_{i=1}^\infty b_i(\lambda) \overline{b_i(0)}.
\end{align*}
By Lemma \ref{backforwshift}, we have
\begin{align*}
\overline{\lambda} K& = M_z^* K = LK + \sum_{i=1}^\infty \langle K, b_i\rangle_{\HH(B)}Lb_i\\
& = \frac{K-K(0)}{z} + \sum_{i=1}^\infty \overline{b_i(\lambda)}\frac{b_i-b_i(0)}{z}\\
& = \frac{K-K(0)}{z} + \frac{\sum_{i=1}^\infty \overline{b_i(\lambda)} b_i - \sum_{i=1}^\infty \overline{b_i(\lambda)} b_i(0)}{z}.
\end{align*}
Since $\sum_{i=1}^\infty \langle K, b_i\rangle_{\HH(B)}Lb_i$ converges in $\HH(B)$, $1 - \sum_{i=1}^\infty b_i(\lambda) \overline{b_i(0)} = \overline{K(0)}$ and $\HH(B)$ is $M_z$-invariant, we have $\sum_{i=1}^\infty \overline{b_i(\lambda)} b_i$ converges in $\HH(B)$.
So
$$K(z) = \frac{K(0) - \sum_{i=1}^\infty \overline{b_i(\lambda)} b_i(z) + \sum_{i=1}^\infty \overline{b_i(\lambda)} b_i(0)}{1-z\overline{\lambda}} = \frac{1 - \sum_{i=1}^\infty \overline{b_i(\lambda)} b_i(z)}{1-z\overline{\lambda}}.$$
Note that
\begin{align*}
0 & = (1-|\lambda|^2)K(\lambda) = e_\lambda ((1-z\overline{\lambda})K)\\
& =e_\lambda\left(1 - \sum_{i=1}^\infty \overline{b_i(\lambda)} b_i\right)\\
& = 1 - \sum_{i=1}^\infty |b_i(\lambda)|^2.
\end{align*}
Thus $\|B(\lambda)\|_{l^2}^2 = \sum_{i=1}^\infty |b_i(\lambda)|^2 = 1$, and so (iv) follows from Theorem \ref{non-tangentiallimit} (ii) and (iii).
\end{proof}

By the above proof, if $\overline{\lambda}$ is an eigenvalue of $M_z^*$, then $\ker e_\lambda = [z-\lambda]$ and $\dim [z-\lambda]^\perp = 1$.
Suppose $\HH(B)$ is $M_z$-invariant. If $f \in \HH(B)$ and $[f] = \HH(B)$, then we say that $f$ is cyclic in $\HH(B)$.
When $B = b \in \HS(\C, \C)$ is a non-extreme rational function, cyclic vectors of $M_z$ were characterized in \cite{GL23, FG}.
When $\dim \HD = 1$, the following result is Theorem 1 in \cite{Berg} which includes the case $B = b \in \HS(\C, \C)$ rational. We restate Theorem C in the following.
\begin{theorem}\label{cyclicvector}
Let $B \in \HS(\HD, \C)$ be such that $\log (1- BB^*) \in L^1(\T)$, and $a$ the mate of $B$. Suppose $\dim \HD < \infty, \dim (aH^2)^\perp < \infty$, and $\overline{\lambda_1}, \overline{\lambda_2}, \ldots, \overline{\lambda_s}$ are the distinct eigenvalues of $M_z^*$ on $(aH^2)^\perp$. Then each $\overline{\lambda_j}$ is in $\T$, and $f \in \HH(B)$ is cyclic in $\HH(B)$ if and only if $f$ is outer and $f(\lambda_j) \neq 0$, $j = 1, 2, \ldots, s$.
\end{theorem}
\begin{proof}
If $\lambda_j \in \D$ for some $j$, then $K_{\lambda_j}^B$ is the eigenvector associated with $\overline{\lambda_j}$ for $M_z^*$. So $K_{\lambda_j}^B \perp aH^2$, and $a(\lambda_j) = 0$, which is a contradiction. Thus each $\overline{\lambda_j}$ is in $\T$. By Lemma \ref{pointevl}, each $f(\lambda_j)$ exists.
If $f \in \HH(B)$ is cyclic in $\HH(B)$, then it is clear that $f$ is outer and $f(\lambda_j) \neq 0$, $j = 1, 2, \ldots, s$.

Now suppose $f$ is outer and $f(\lambda_j) \neq 0$, $j = 1, 2, \ldots, s$. Note that $a$ is a multiplier of $\HH(B)$. In fact, by Lemma \ref{containment}, $\HM(a) \subseteq \HH(B)$. So $a\HH(B) \subseteq \HM(a) \subseteq \HH(B)$ and $a$ is a multiplier of $\HH(B)$.
By \cite[Theorem 5.11]{AM19}, if $\phi \in [f] \ominus z[f]$ is of norm $1$, then
$$[f] = \{g \in \HH(B): g/\phi \in H^2, (g/\phi)\phi^+ \in H^2(\HD)\},$$
where $\phi^+$ is the unique function satisfying Theorem \ref{matematrix} (ii). Thus $a[f] \subseteq [f]$. Now we show $aH^2 \subseteq [f]$. Let $\varepsilon >0 $ and let $h \in H^2$. Since $f$ is outer, there exists a polynomial $q$ such that
$$\|h -qf\|_{H^2} < \frac{\varepsilon}{2}.$$
Using that $a[f] \subset [f]$, we see that $aqf \in [f]$, and so there exists a polynomial $p$ such that
$$\|aqf - pf\|_{\HH(B)} < \frac{\varepsilon}{2}.$$
Thus it follows that
\begin{align*}
\|ah - pf\|_{\HH(B)}& \leq \|ah - aqf\|_{\HH(B)} + \|aqf - pf\|_{\HH(B)}\\
& \leq \|h -qf\|_{H^2} + \|aqf - pf\|_{\HH(B)}\\
& < \frac{\varepsilon}{2} + \frac{\varepsilon}{2} = \varepsilon.
\end{align*}
Thus $aH^2 \subset [f]$.
It is left to show that $(aH^2)^\perp \subseteq [f]$. Since polynomials are dense in $\HH(B)$, the eigenspaces of $M_z^*$ are one dimensional, and so the minimal and characteristic polynomials of $M_z^*|(aH^2)^\perp$ coincide. Let
$$\HN_{\lambda_j} : = \bigvee_{n\ge 0} \ker (M_z^*-\overline{\lambda_j})^n$$
be the root subspace of $M_z^*$ corresponding to the eigenvector $\overline{\lambda_j}$. Then it is clear that $(aH^2)^\perp=\bigvee_{j=1}^s \HN_{\lambda_j}$.
Note that $[f]^\perp$ is an $M_z^*$-invariant subspace contained in $(aH^2)^\perp$. If there is $h \in \bigvee_{j=1}^s \HN_{\lambda_j}\setminus \{0\}$ such that $h \in [f]^\perp$, we then claim that there is some $j$ such that $K_{\lambda_j}^B \in [f]^\perp$, where $\langle f, K_{\lambda_j}^B\rangle_{\HH(B)} = f(\lambda_j)$. But this then contradicts that $f(\lambda_j) \neq 0$. We thus obtain that $[f] = \HH(B)$.

Now we prove the claim. Suppose $\dim \HN_{\lambda_j} = m_j, K_{\lambda_j}^l \in \ker(M_z^*-\overline{\lambda_j})^l$ and $K_{\lambda_j}^l \not\in \ker(M_z^*-\overline{\lambda_j})^{l-1}, K_{\lambda_j}^1 = K_{\lambda_j}^B, j =1, \ldots, s$, this is possible since the minimal and characteristic polynomials of $M_z^*|(aH^2)^\perp$ coincide. Then there are $c_{j,l} \in \C$ such that $h = \sum_{j=1}^s\sum_{l=1}^{m_j} c_{j,l} K_{\lambda_j}^l$. Suppose there is some $j$ such that $c_{j,l} \neq 0$. Then
$$0 \neq \Pi_{k \neq j} (M_z^* - \overline{\lambda_k})^{m_k} h \in [f]^\perp \cap \HN_{\lambda_j}.$$
So we may assume that $h = \sum_{l=1}^{m_j} c_l K_{\lambda_j}^l$ and $c_m \neq 0$, $c_l = 0, l > m$. Then $$(M_z^* -\overline{\lambda_j})^{m-1} h = c_m (M_z^* -\overline{\lambda_j})^{m-1} K_{\lambda_j}^m$$
which is a nonzero constant multiple of $K_{\lambda_j}^B$. This proves the claim and finishes the proof.
\end{proof}
By the above proof, we can actually determine all the $M_z$-invariant subspaces of $\HH(B)$ which contains $(aH^2)^\perp$.
When $B$ is a rational Schur function, the $M_z$-invariant subspaces of $\HH(B)$ are characterized by \cite[Theorem 3.10]{GL23}. Note that \cite[Theorem 3.10]{GL23} is stated for scalar rational functions $b$, but it holds for rational Schur functions $B$, see \cite[Theorem 7.2]{LGR}.

\

\noindent \textbf{Acknowledgements.}
S. Luo was supported by NNSFC (12271149), Natural Science Foundation of Hunan Province (2024JJ2008).
The authors thank the referee for careful reading of the paper and many detailed comments which greatly improve the presentation of this paper.


\begin{thebibliography}{99}
\bibitem{AM19}
A. Aleman, B. Malman, Hilbert spaces of analytic functions with a contractive backward shift, J. Funct. Anal. 277, no. 1, 157-199 (2019).

\bibitem{BallBoloBasics}
J.A. Ball, V. Bolotnikov, De Branges-Rovnyak spaces: Basics and Theory,
Operator Theory, Springer (2015), pp. 631-679.

\bibitem{Berg}
A. Bergman, On cyclicity in de Branges-Rovnyak spaces, preprint, arXiv:2309.09601

\bibitem{dBR66}
L. de Branges, J. Rovnyak, Square summable power series. Holt, Rinehart and Winston, New York (1966).

\bibitem{Do66}
R.G. Douglas, On majorization, factorization, and range inclusion of operators on Hilbert space, Proc. Amer. Math. Soc. 17 (1966), 413-415.

\bibitem{FG}
E. Fricain, S. Grivaux, Cyclicity in de Branges-Rovnyak spaces, preprint, arXiv:2210.15735

\bibitem{FM16}
E. Fricain, J. Mashreghi, The theory of $\mathcal{H}(b)$
spaces. Vol. 1, New Mathematical Monographs 20, Cambridge University Press,
Cambridge, 2016.

\bibitem{FM162}
E. Fricain, J. Mashreghi, The theory of $\mathcal{H}(b)$
spaces. Vol. 2, New Mathematical Monographs 21, Cambridge University Press,
Cambridge, 2016.

\bibitem{FMR20}
E. Fricain, J. Mashreghi, R. Rupam, Backward shift invariant subspaces in reproducing kernel Hilbert spaces, Math. Scand. 126 (2020), no. 1, 142-160.

\bibitem{GL23}
C. Gu, S. Luo, Higher order isometric shift operator on the de Branges-Rovnyak space, J. Operator Theory 89 (2023), no. 1, 105-123.

\bibitem{LMT}
C. Liaw, R.T.W. Martin, S. Treil,, Matrix-valued Aleksandrov-Clark measures and Carath\'{e}odory angular derivatives, J. Funct. Anal. 280 (2021), no. 3, Paper No. 108830, 33 pp.

\bibitem{LGR}
S. Luo, C. Gu, S. Richter, Higher order local Dirichlet integrals and de Branges-Rovnyak spaces. Adv. Math. 385 (2021), Paper No. 107748, 47 pp.

\bibitem{RR97}
M. Rosenblum, J. Rovnyak, Hardy classes and operator theory, Dover Publications, Inc., Mineola, NY, 1997.

\bibitem {Sa86}
D. Sarason, Doubly shift-invariant spaces in $H^{2}$, J. Operator Theory 16, no. 1, 75-97 (1986).

\bibitem {Sa88}
D. Sarason, Angular derivatives via Hilbert space, Complex Variables Theory Appl. 10 (1988), no. 1, 1-10.

\bibitem{Sa94}
D. Sarason, Sub-Hardy Hilbert Spaces in the Unit Disk. Wiley, New York (1994).

\end{thebibliography}
\end{document}